\documentclass[12pt]{amsart}

\usepackage{amsmath}
\usepackage{amssymb}
\usepackage{bm}
\usepackage{graphicx}
\usepackage{psfrag}
\usepackage{color}
\usepackage{hyperref}
\hypersetup{colorlinks=true, linkcolor=blue, citecolor=magenta, urlcolor=wine}
\usepackage{url}
\usepackage{algpseudocode}
\usepackage{fancyhdr}
\usepackage{mathtools}
\usepackage{tikz-cd}
\usepackage{xy}
\input xy
\xyoption{all}
\usepackage{stmaryrd}
\usepackage{calrsfs}

\newtheorem*{maintheorem*}{Main Theorem}
\newtheorem{theorem}{Theorem}[section]
\newtheorem{prop}[theorem]{Proposition}

\newtheorem{lem}[theorem]{Lemma}

\theoremstyle{definition}
\newtheorem{defn}[theorem]{Definition}
\newtheorem{rem}[theorem]{Remark}
\newtheorem{example}[theorem]{Example}

\numberwithin{equation}{section}

\newcommand{\nn}{\mathbb{N}}

\newcommand{\ppp}{\mathcal{P}}
\newcommand{\qq}{\mathbb{Q}}

\providecommand\ldb{\llbracket}
\providecommand\rdb{\rrbracket}

\newcommand{\ppf}{\mathcal{P}_{\text{fin}}}
\newcommand{\ppx}{\mathcal{P}_{\text{fin}, 0}}



\begin{document}
	
	\mbox{}
	\title{On the atomicity of power monoids of Puiseux monoids}
	
	\author{Victor Gonzalez}
	\address{Department of Mathematics\\Miami Dade College\\Miami, FL 33135, USA}
	\email{victorm.gonzalez0202@gmail.com}
	
	\author{Eddy Li}
	\address{The Nueva School, San Mateo, CA 94403, USA}
	\email{eddylihere@gmail.com}
	
	\author{Henrick Rabinovitz}
	\address{Melrose School\\Melrose, MA 02176, USA}
	\email{hrey77@gmail.com}
	
	\author{Pedro Rodriguez}
	\address{School of Mathematics and Statistics\\Clemson University\\Clemson, SC 29634, USA}
	\email{pedror@clemson.edu}
	
	\author{Marcos Tirador}
	\address{Matcom, Universidad de La Habana, Habana 10400, Cuba}
	\email{marcosmath44@gmail.com}

\date{\today}

\begin{abstract}
	A submonoid of the additive group $\qq$ is called a Puiseux monoid if it consists of nonnegative rationals. Given a monoid $M$, the set consisting of all nonempty finite subsets of $M$ is also a monoid under the Minkowski sum, and it is called the (finitary) power monoid of $M$. In this paper we study atomicity and factorization properties in power monoids of Puiseux monoids. We specially focus on the ascent of the property of being atomic and both the bounded and the finite factorization properties (the ascending chain on principal ideals and the length-finite factorization properties are also considered here). We prove that both the bounded and the finite factorization properties ascend from any Puiseux monoid to its power monoid. On the other hand, we construct an atomic Puiseux monoid whose power monoid is not atomic. We also prove that the existence of maximal common divisors for nonempty finite subsets is a sufficient condition for the property of being atomic to ascend from a Puiseux monoid to its power monoid.
\end{abstract}

\bigskip
\maketitle

\bigskip
%%%%%%%%%%%%
%%%%%%%%%%%%
\section{Introduction}
\label{sec:intro}

Let $M$ be an (additive) monoid. For nonempty subsets $S$ and $T$ of $M$, the set $S+T := \{s+t \mid s \in S \text{ and } t \in T\}$ is called the Minkowski sum of $S$ and $T$ in $M$. It is clear that the set $\ppp(M)$ consisting of all nonempty subsets of $M$ is also a monoid under the Minkowski sum, which is called the power monoid of $M$. In addition, the submonoid of $\ppp(M)$ consisting of all its nonempty finite subsets is called the finitary power monoid of $M$ and is denoted by $\ppf(M)$. Power monoids and finitary power monoids were systematically investigated in the second half of the past century by Tamura et al. (see~\cite{tT86} and references therein). Two of the most investigated problems in the context of power monoids were the ascent of algebraic properties from the monoid $M$ to $\ppp(M)$ and the isomorphism problem (that is, whether the fact that $\ppp(M)$ and $\ppp(M')$ are isomorphic in certain prescribed class of monoids guarantees that $M$ and $M'$ are also isomorphic). For more into these two problems and further progress on the study of power monoids, see~\cite{jP86} and references therein.
\smallskip

A submonoid of the additive group $\qq$ is called a Puiseux monoid if it consists of nonnegative rationals. Puiseux monoids have played a crucial role to construct needed counterexamples in commutative ring theory; see, for instance, \cite[Theorem~1.3]{aG74} and \cite[Example~1.1]{AAZ90} (more recent applications of Puiseux monoids to commutative ring theory can be found in~\cite[Section~4]{CG19}, \cite[Section~3]{GL23a}, and \cite[Main Theorem]{GR23}). However, it was not until recently that a systematic investigation of the atomic structure of Puiseux monoids was initiated by Gotti~\cite{fG17}. Since then, Puiseux monoids have been actively investigated from the factorization-theoretical viewpoint (see \cite{GGT21} and references therein). For a survey on the atomicity and factorizations of Puiseux monoids, see~\cite{CGG20a}.
\smallskip

The central algebraic objects that we study in this paper are finitary power monoids of Puiseux monoids. For simplicity, from now on we will use the simpler term ``power monoid" to refer to ``finitary power monoids" (this should not cause any confusion because we do not deal at all with non-finitary power monoids). Additive monoids consisting of nonnegative integers, also known as numerical monoids, account up to isomorphism for all finitely generated Puiseux monoids. Factorizations and certain arithmetical aspects of the power monoid $\ppf(\nn_0)$ were studied by Fan and Tringali in~\cite{FT18}, while atomic and ideal-theoretical aspects of power monoids of numerical monoids were recently studied by Bienvenu and Geroldinger in~\cite{BG23}. Most recently, the isomorphism problem in the class consisting of all the (restricted) power monoids of Puiseux monoids was settled in~\cite{TY23} by Tringali and Yan. The primary problem we consider in this paper is the ascent of certain atomic and factorization properties from Puiseux monoids to their corresponding power monoids.
\smallskip

A commutative monoid is called atomic provided that every non-invertible element can be written as a finite sum of atoms (i.e., irreducible elements). The ascent of atomicity from the class of Puiseux monoids to that of their monoid algebras was investigated by Coykendall and Gotti in~\cite{CG19} and, more recently, by Gotti and Rabinovitz in~\cite{GR23} (the ascent of atomicity from the class of torsion-free commutative monoids to that of their monoid algebras was first posed as an open problem by Gilmer in~\cite[page 189]{rG84}). In Section~\ref{sec:atomicity}, we tackle the ascent of atomicity from the class of Puiseux monoids to that of their power monoids: we prove that, in general, the property of being atomic does not ascend to power monoids of Puiseux monoids. On the other hand, we prove that the same property does ascend if we restrict to the class of MCD Puiseux monoids (a commutative monoid is called an MCD monoid if every nonempty finite set has a maximal common divisor). In Section~\ref{sec:atomicity}, we also show that the ascending chain condition on principal ideals (ACCP) ascends to power monoids of Puiseux monoids (the ACCP has been largely investigated in connection to atomicity: see the recent papers~\cite{GL23,GL23a} and references therein).
\smallskip

Following~\cite{AAZ90}, we say that an atomic monoid is a bounded factorization monoid provided that every non-invertible element has a finite set of factorization lengths (the length of a factorization is the number of atoms it has, counting repetitions). Following the same paper, we say that an atomic monoid is a finite factorization monoid provided that every non-invertible element has only finitely many factorizations (clearly, every finite factorization monoid is a bounded factorization monoid). Both the bounded and the finite factorization properties have been actively investigated since they were introduced in the context of integral domains by Anderson, Anderson, and Zafrullah in~\cite{AAZ90} and generalized to the context of cancellative commutative monoids by Halter-Koch in ~\cite{fHK92} (see the recent survey~\cite{AG22} and references therein). The main purpose of Section~\ref{sec:BF and FF} is to consider the ascent of both the bounded and the finite factorization properties. We prove that both properties ascend to power monoids of Puiseux monoids. As the bounded factorization property, the length-finite factorization property, recently introduced by Geroldinger and Zhong in~\cite{GZ20}, is a weaker notion of the finite factorization property. We argue in Section~\ref{sec:BF and FF} that the length-finite factorization property does not ascend to power monoids of Puiseux monoids.

\bigskip
%%%%%%%%%%%%
%%%%%%%%%%%%
\section{Background}
\label{sec:background}

In the background section of this paper, we will be covering key concepts about commutative monoids, which are essential for understanding the results we will discuss later. For any terms or notations on commutative monoids not defined here, the reader can consult~\cite{pG01}. Additionally, \cite{GZ20} offers a recent survey of factorization in commutative monoids.

\medskip
%%%%%%%%%%%%%%%%
\subsection{General Notation}

We define the set of natural numbers, excluding zero, as $\mathbb{N} := \{1, 2, \ldots\}$, and when including zero, it is represented as $\mathbb{N}_0 := \mathbb{N} \cup \{0\}$. For any two integers $x, y \in \mathbb{Z}$, the interval of integers between them is denoted by $\ldb x, y \rdb$, which is an empty set if $x > y$. Furthermore, for any subset $X$ of $\mathbb{R}$ and a real number $r$, the set $X_{\geq r}$ is conformed by all elements $x$ in $X$ that are greater than or equal to $r$. Similarly, $X_{> r}$ is used to represent elements strictly greater than $r$.

For every $q \in \mathbb{Q}_{> 0}$, there exist unique natural numbers $n$ and $d$ such that $q = n/d$ and $\gcd(n,d) = 1$, which we denote as $\mathsf{n}(q)$ and $\mathsf{d}(q)$, respectively. We call $\mathsf{n}(q)$ and $\mathsf{d}(q)$ the \emph{numerator} and \emph{denominator} of $q$, respectively. The $p_d$-adic valuation for a nonzero integer $n$ is the highest value in the set $\{ k \in \mathbb{N} \mid d^k \text{ divides } n \}$, denoted by $\mathsf{v}_d(n)$. For any rational number $q$, its $p_d$-adic valuation is defined as $\mathsf{v}_d(q) = \mathsf{v}_d(\mathsf{n}(q)) - \mathsf{v}_d(\mathsf{d}(q))$.

\medskip
%%%%%%%%%%%%%%%%%%%
\subsection{Commutative Monoids}

In this paper, a \emph{monoid} is understood as a commutative semigroup with an identity element. Let $M$ be a monoid. As we are assuming every monoid here is commutative, we will be using additive notation, unless otherwise is specified. In particular, ``+'' denotes the operation of $M$, while $0$ denotes the identity element. We define $M^\bullet$ as $M$ excluding the zero element, and $M$ is called \emph{trivial} if $M^\bullet$ is empty. The group of invertible elements within $M$ is represented as $M^{\times}$. The \emph{reduced monoid} of $M$, denoted $M_{\text{red}}$, is the quotient of $M$ by $M^{\times}$. In addition, we say that $M$ is \emph{reduced} if $M^\times$ is the trivial group, in which case we identify $M_{\text{red}}$ with $M$. 
%The \textit{difference group} of $M$, denoted $\gp(M)$, is the unique abelian group up to isomorphism such that any abelian group containing an isomorphic image of $M$ must also contain an  isomorphic image of $\gp(M)$. The \textit{rank} of $M$ is defined as the rank of $\gp(M)$ as a module over $\mathbb{Z}$, namely, the dimension of $\mathbb{Q} \otimes_\mathbb{Z} \gp(M)$. 
The monoid $M$ is said to be \emph{torsion-free} if $nr = ns$ for some $n \in \mathbb{N}$ and $r,s \in M$ implies $r = s$. %A monoid is torsion-free if and only if its difference group is torsion-free  (refer to \cite[Section~2.A]{BG09}). 
An element $t \in M$ is called \emph{cancellative} if for all $r,s \in M$, the equality $r+t = s+t$ implies that $r=s$, and the whole monoid $M$ is called \emph{cancellative}  We say that the monoid $M$ is \emph{cancellative} if every element of $M$ is cancellative. On the other hand, the monoid $M$ is called \emph{unit-cancellative} if for all $r,s \in M$, the equality $r+s = r$ guarantees that $s \in M^\times$. It follows from the given definitions that every cancellative monoid is unit-cancellative.

For elements $r, s \in M$, $s$ is said to \textit{divide} $r$ in $M$ if there exists $t \in M$ such that $r = s + t$; in this case, we write $s \mid_M r$. A submonoid $M'$ of $M$ is called a \textit{divisor-closed submonoid} if any element of $M$ dividing an element of $M'$ is in  $M'$ as well. Given a subset $S$ of $M$, the smallest submonoid of $M$ containing $S$ is denoted by $\langle S \rangle$, and in this case, $S$ is called a \textit{generating set} of $\langle S \rangle$. The monoid $M$ is \textit{finitely generated} if it can be expressed as $M = \langle S \rangle$ for some finite subset $S$ of $M$.

An element $a$ in $M \setminus M^{\times}$ is an \emph{atom} if for $a = r + s$ with $r,s \in M$, either $r$ or $s$ belongs to $M^{\times}$. The set of all atoms in $M$ is denoted by $\mathcal{A}(M)$. In a reduced monoid $M$, we note that $\mathcal{A}(M)$ is included in any generating set of $M$. An element $b \in M$ is called \emph{atomic} if it is in the submonoid of $M$ generated by $\mathcal{A}(M)$. Following \cite{pC68}, we say that $M$ is \emph{atomic} if every non-invertible element in $M$ is atomic.

We say that $d \in M$ is a \emph{maximal common divisor} (or \emph{mcd}) of $A \subset M$ if $d$ divides every element in $A$ and does not exists a non-invertible element $d' \in M$ such that $d' + d$ also divides every element in $A$. Notice that not every subset $A$ of $M$ has an mcd. We call $M$ a k-MCD monoid, for some $k \in \nn$, if every subset $A \subset M$ such that $|A| = k$, has a maximal common divisor. Furthermore, we call $M$ an MCD-monoid if $M$ is k-MCD for every $k \in \nn$.

A subset $I$ of $M$ is an \emph{ideal} if the set
\[
	I + M := \{r + s \mid r \in I \text{ and } s \in M\}
\]
is a subset of $I$. An ideal of $M$ is called \textit{proper} if it is strictly contained in $M$. The ideals of the form $b + M := \{b\} + M$ for some $b \in M$ are called \emph{principal ideals}. We say that $M$ satisfies the \emph{ascending chain condition on principal ideals} (ACCP) if every ascending chain $(r_n + M)_{n \in \nn}$ of principal ideals of $M$ eventually stabilizes; that is, there is an $n_0 \in \nn$ such that $r_n + M = r_{n+1} + M$ for every $n \ge n_0$. It follows from~\cite[Theorem~2.28]{FT18} that if a unit-cancellative monoid satisfies the ACCP, then it must be atomic.

\medskip
%%%%%%%%%%%%%%%
\subsection{Factorizations}

A multiplicative monoid $F$ is said to be \emph{free on} a subset $A$ if every element $x \in F$ has a unique representation in the form
\[
	x = \prod_{a \in A} a^{\mathsf{v}_a(x)},
\]
where $\mathsf{v}_a(x) \in \nn_0$ and $\mathsf{v}_a(x) > 0$ only for finitely many $a \in A$. For any given set $A$, there is a unique free monoid $F$ on $A$, up to isomorphism. The free monoid formed on $\mathcal{A}(M_{\text{red}})$, labeled $\mathsf{Z}(M)$, is called the \emph{factorization monoid} of $M$, and its elements are known as \emph{factorizations}. For a factorization $z = a_1 \cdots a_n$ in $\mathsf{Z}(M)$, $n$ is called the \emph{length} of $z$, and is denoted by $|z|$. The monoid homomorphism from $\mathsf{Z}(M)$ to $M_{\text{red}}$, which takes each atom to itself, is called the \emph{factorization homomorphism} of $M$. The set of all factorizations of an element $x \in M$, named $\mathsf{Z}(x)$, is the inverse image of $x$ under the factorization homomorphism. Observe that $M$ is atomic if and only if every element in $M$ has at least one factorization. For each $x \in M$, the \emph{set of lengths} of $x$ is defined by
\[
	\mathsf{L}(x) := \mathsf{L}_M(x) := \{|z| :  z \in \mathsf{Z}(x)\}.
\]
The monoid $M$ is called a \emph{UFM} (\emph{unique factorization monoid}) if each element has exactly one factorization, an \emph{FFM} (\emph{finite factorization monoid}) if each element has a finite number of factorizations, and an \emph{HFM} (\emph{half-factorial monoid}) if every factorization of the same element has the same length. A monoid $M$ is called \emph{BFM} (\textit{bounded factorization monoid}) if it has a finite upper bound on the lengths of factorizations for each of its elements. It follows from~\cite[Theorem~2.28 and Corollary~2.29]{FT18} that if a unit-cancellative monoid is a BFM, then it must satisfy the ACCP. We call the monoid $M$ an \emph{LFM} (\emph{length-factorial monoid}) if every element of $M$ has at most one factorization of length $k$ for each $k \in \nn$. Observe that if $M$ is both HFM and LFM, then $M$ is a UFM. Additionally, we say that $M$ is an \textit{LFFM} (\emph{length-finite factorization monoid}) if the set $\{z \in \mathsf{Z}(x) : |z| = k\}$ is finite for every $x \in M$ and $k \in \nn$. Observe that if the monoid $M$ is an LFFM and a BFM at the same time, then it is an FFM.

\medskip
%%%%%%%%%%%%%%%%%%%%%%%%
\subsection{Numerical and Puiseux Monoids}

A \emph{numerical monoid}\footnote{Numerical monoids are often called ``numerical semigroups" in the literature.} is a submonoid of $(\mathbb{N}_0,+)$ characterized by having a finite complement in $\mathbb{N}_0$. For a numerical monoid $N$ different from $\nn_0$, its largest element is referred to as the \emph{Frobenius number} of $N$, denoted as $F(N)$. Numerical monoids are finitely generated and, therefore, FFMs (see \cite[Proposition~2.7.8]{GH06}). Numerical monoids have been the subject of an extensive study, revealing connections across various mathematical fields and practical applications (see \cite{GR09, AG16} for a comprehensive overview).

On the other hand, a \emph{Puiseux monoid} is a submonoid of $(\mathbb{Q}_{\ge 0},+)$. %, presenting a distinct set of properties. 
In contrast to numerical monoids, Puiseux monoids may be neither finitely generated nor atomic, as it is the case of $\qq_{\ge 0}$, whose set of atoms is empty. Interestingly, some atomic Puiseux monoids are not BFMs, such as $\langle 1/p \mid p \in \mathbb{P} \rangle$, while others can be BFMs but not FFMs, like $\{0\} \cup \mathbb{Q}_{\ge 1}$. The exploration of Puiseux monoids, particularly in factorization theory and their relevance in both commutative and non-commutative contexts, has become a subject of recent scholarly interest (see \cite{CGG20, CGG20a, CG19, aG74, BG20}).

We say that a sequence of rational numbers is \emph{well-ordered} (resp., \emph{co-well-ordered}) if it contains no strictly decreasing (resp., increasing) subsequence. Following~\cite{hP22}, we say that a Puiseux monoid $M$ is \emph{well-ordered} (resp., \emph{co-well-ordered}) if it can be generated by a well-ordered (resp., co-well-ordered) sequence. There are atomic Puiseux monoids that are neither well-ordered nor co-well-ordered (see \cite[Example~3.1]{JKK23}).

\medskip
%%%%%%%%%%%%%%%%%%%%%%%%%%%%%%%%
\subsection{Power Monoids and Restricted Power Monoids}

Let $M$ be a monoid. The concept of power monoids emerges from considering the set conformed by all nonempty finite subsets of a monoid $M$, which we denote as $\mathcal{P}_{\text{fin}}(M)$. For any two elements $S, T \in \mathcal{P}_{\text{fin}}(M)$, their sum is defined by
\[
	S +' T := \{s+t \mid s \in S \text{ and } t \in T \}.
\]
This operation endows $\mathcal{P}_{\text{fin}}(M)$ with a monoid structure, hereafter referred to as $+$ for simplicity  (this will hardly cause any ambiguity).

\begin{defn}
	For a monoid $M$, the \emph{power monoid} of $M$ is the monoid $\mathcal{P}_{\text{fin}}(M)$ under the sum operation defined before.
\end{defn}
Further, we introduce $\mathcal{P}_{\text{fin}, \times}(M)$, which is the subset of $\mathcal{P}_{\text{fin}}(M)$ consisting of those elements intersecting with $M^{\times}$; that is,
\[
	\mathcal{P}_{\text{fin}, \times}(M) := \{S \in \mathcal{P}_{\text{fin}}(M) \mid S ~\cap~ M^{\times} \text{ is nonempty} \}.
\]
This subset forms a submonoid, known as the restricted power monoid of $M$.

\begin{defn}
	For a monoid $M$, the \emph{restricted power monoid} of $M$ is the submonoid $\mathcal{P}_{\text{fin}, \times}(M)$ of  $\mathcal{P}_{\text{fin}}(M)$.
\end{defn}

One can readily check that $\ppf$ and $	\mathcal{P}_{\text{fin}, \times}(M)$ are cancellative if and only if $M$ is the trivial monoid. In addition, it follows from \cite[Proposition~3.5]{FT18} that $\ppf$ and $\mathcal{P}_{\text{fin}, \times}(M)$ are unit-cancellative when $M$ is a totally ordered monoid, so the power monoid and the restricted power monoid of any given Puiseux monoid are unit-cancellative.

\bigskip
%%%%%%%%%%%%%%%%%%%%%%%%%%%%%%%
%%%%%%%%%%%%%%%%%%%%%%%%%%%%%%%
\section{Atomicity in Power Monoids of Puiseux Monoids}
\label{sec:atomicity}

The main goal of this section is to investigate when the property of being atomic ascends from a Puiseux monoid to its corresponding power monoid (i.e., to investigate under which conditions the power monoid of an atomic Puiseux monoid is also atomic). First, we argue the following lemma.

\begin{lem}\label{lem:size of the sum}
	Let $M$ be a Puiseux monoid. Let $B$ be an element of $\mathcal{P}_{\emph{fin}}(M)$. Then the following statements hold.
	\begin{enumerate}
		\item If $|B|=1$, then $|B + C| = |C|$ for every $C \in \mathcal{P}_{\emph{fin}}(M)$.
		\smallskip
		
		\item If $|B| \ge 2$, then $|B + C| > |C|$ for every $C \in \mathcal{P}_{\emph{fin}}(M)$.
	\end{enumerate}
\end{lem}

\begin{proof}
	(1) This follows immediately.
	\smallskip
	
	(2) Let $b_1$ and $b_2$ be elements of $B$, with $b_1 < b_2$. We have that $|\{b_1\} + C| = |C|$ since Puiseux monoids are cancellative. Let $c$ be the maximum element of $C$. Then $c + b_2 > c + b_1$ and, therefore, $c + b_2$ is not in $\{b_1\} + C$. Hence $|B + C| \ge |\{b_1\} + C| + 1 > |C|$, which concludes the proof.
\end{proof}

It is hardly a surprise that atomicity does not ascend from a Puiseux monoid to its power monoid (as atomicity does not ascend neither to polynomial extensions~\cite{mR93} nor to monoid algebras over fields~\cite{CG19}). However, it does if we impose the existence of maximal common divisors.

\begin{theorem} \label{thm:ascend of atomicity}
	For an atomic Puiseux monoid $M$, the following statements hold.
	\begin{enumerate}
		\item If $M$ is an MCD-monoid, then $\ppp_{\emph{fin}}(M)$ is atomic.
		\smallskip
		
		\item If $M$ is not $2$-MCD, then $\ppp_{\emph{fin}}(M)$ is not atomic.
	\end{enumerate} 
\end{theorem}

\begin{proof}
	 Let $M$ be an atomic Puiseux monoid and, or simplicity, set $\mathcal{P} :=  \mathcal{P}_{\text{fin}}(M)$. 
	 \smallskip
	 
	(1) Suppose that $M$ is an MCD-monoid. Assume, by way of contradiction, that $\mathcal{P}$ is not atomic, and then fix a non-atomic element $B_0$ of $\ppp$. Since $B_0$ is non-atomic, there exists a strictly descending chain of non-atomic elements of $\ppp$ starting at $B_0$, that is, a sequence $(B_n)_{n \in \nn_0}$ of non-atomic elements of $\mathcal{P}$ such that $B_{n + 1} \mid_\mathcal{P} B_n$ and $B_{n + 1} \neq B_n$ for every $n \in \mathbb{N}_0$. Now Lemma~\ref{lem:size of the sum} guarantees the existence of $N \in \mathbb{N}$ such that $k := |B_{n}| = |B_N|$ for every $n \ge N$. Among all the strictly descending chains of non-atomic elements of $\mathcal{P}$ starting at $B_0$, assume that $(B_n)_{n \ge 0}$ is on minimizing $k$. Thus, every divisor of $B_N$ in $\mathcal{P}$ has cardinality either $1$ or $k$. Write $B_N = \{b_1, b_2, \dots, b_k\}$. Let $d$ be a maximal common divisor of $b_1, b_2, \dots, b_k$ in $M$. Then the only common divisor of $A := \{b_1 - d, b_2 - d, \dots, b_k - d \}$ is $\{0\}$. The Lemma~\ref{lem:size of the sum}, along with the minimality of $k$, ensures that in any decomposition of $A$ as a sum of two elements in $\mathcal{P}$, one of them must be a singleton, and so one of them is $\{0\}$. Thus, $A$ is an atom of $\mathcal{P}$. In addition, observe that $\{d\}$ is an atomic element in $\mathcal{P}$ because $d$ is atomic in $M$. Since $B_N = A + \{d\}$, the fact that both $A$ and $\{d\}$ are atomic elements of $\mathcal{P}$ implies that $B_N$ is atomic, which is a contradiction. Hence $\mathcal{P}$ must be atomic.
	\smallskip
	
	(2) Suppose that $M$ is not 2-MCD. Assume, by way of contradiction, that $\mathcal{P}$ is atomic. Let $b_1, b_2$ be two elements of $M$ that do not have a maximal common divisor. Consider the element $B := \{b_1, b_2\}$ of $\mathcal{P}$. Because $0$ is not a maximal common divisor of $b_1$ and $b_2$ in $M$, we can take a nonzero $d \in M$ such that $d \mid_M b_1$ and $d \mid_M b_2$. Thus, $B = \{b_1-d, b_2-d\} +\{d\}$, and so $B$ is not an atom of $\mathcal{P}$. As $\mathcal{P}$ is atomic, we can write $B = A_1 + A_2 + \dots +A_k$ for some $k \ge 2$ and atoms $A_1, A_2, \dots, A_k$ of $\mathcal{P}$. Since $|B| = 2$, there exists an element among $A_1, \dots, A_k$ with size at least $2$. Then we can assume that $|A_1| \ge 2$. Thus, it follows from Lemma~\ref{lem:size of the sum} that
	\[
		2 = |B| = |A_1 + (A_2 + \cdots + A_k)| > |A_2 + \cdots + A_k|,
	\]
	and so $|A_i| = 1$ for every $i \in \ldb 2,k \rdb$, and so $|A_1| = |B| = 2$. Write $A_i = \{a_i\}$ for every $i \in \ldb 2,k \rdb$, and set $a := a_2 + \dots + a_k$. Then $A_1 = \{b_1 - a, b_2 - a\}$. Since $b_1$ and $b_2$ have no maximal common divisor, there exists a nonzero $d' \in M$ such that $d' \mid_M b_1 - a$ and $d' \mid_M b_2 - a$. Then $A_1 = \{b_1 - a - d', b_2 - a - d'\} + \{d'\}$, which contradicts that $A_1$ is an atom of $\mathcal{P}$.
\end{proof}

As a consequence, showing that the property of being atomic does not ascend to power monoids in the class of Puiseux monoids amounts to constructing an atomic Puiseux monoid that is not a  2-MCD monoid. We will do so in the following example.
	
\begin{example} \label{ex:atomic PM not 2-MCD}
	We will exhibit an atomic Puiseux monoid $M$ that is not a $2$-MCD monoid.
	
	First, let us inductively construct a sequence of prime numbers with certain desired properties. Take $p_0 = 17$, and then suppose that for $n \in \mathbb{N}_0$, we have chosen primes $p_0, \dots, p_{3n}$ such that $p_i > 15 \cdot 2^i$ for every $i \in \ldb 0, 3n \rdb$. Now take primes $p_{3n+1}, p_{3n+2}$, and $p_{3n+3}$ such that $p_{3n+3} > p_{3n+2} > p_{3n+1} > \max\{p_{3n}, 15 \cdot 2^{3n+3} \}$ and
	\begin{equation} \label{eq:I}
		p_{3n+1} > \max \bigg\{ \mathsf{n}\bigg(  \frac45 - \sum_{i=0}^n \frac1{p_{3i}} \bigg), \mathsf{n}\bigg( \frac67 - \sum_{i=0}^n \frac1{p_{3i}} \bigg) \bigg\}.
	\end{equation}
	Then, we have inductively constructed a strictly increasing sequence of primes $(p_n)_{n \ge 0}$ such that for every $n \in \mathbb{N}_0$, both the inequality~\eqref{eq:I} and $p_n > 15 \cdot 2^n$ hold. Therefore
	\begin{equation} \label{eq:II}
		\sum_{n=0}^\infty \frac 1{p_n} < \frac{1}{15} \sum_{n=0}^\infty \frac{1}{2^n} = \frac 2{15} < \frac 17.
	\end{equation}
	Now, let $M$ be the additive submonoid of $\mathbb{Q}$ defined as follows:
	\[
		M := \bigg\langle \frac{1}{p_{3n}}, \ \frac{1}{p_{3n+1}} \bigg( \frac45 - \sum_{i=0}^n \frac1{p_{3i}} \bigg), \frac{1}{p_{3n+2}} \bigg( \frac67 - \sum_{i=0}^n \frac1{p_{3i}} \bigg) \ \Big{|} \ n \in \mathbb{N}_0 \bigg\rangle.
	\]
	It follows from~\eqref{eq:II} that $\sum_{n=0}^\infty \frac 1{p_n} < \frac 45$ and, therefore, $M$ is a Puiseux monoid. For every $n \in \mathbb{N}_0$, set
	\[
		a_n := \frac{1}{p_{3n}}, \quad b_n :=  \frac{1}{p_{3n+1}} \bigg( \frac45 - \sum_{i=0}^n \frac1{p_{3i}} \bigg), \quad \text{and} \quad c_n :=  \frac{1}{p_{3n+2}} \bigg( \frac67 - \sum_{i=0}^n \frac1{p_{3i}} \bigg).
	\]
	We claim that $M$ is atomic with $\mathcal{A}(M) = \{a_n, b_n, c_n \mid n \in \mathbb{N}_0\}$. To prove our claim, it suffices to show that $\{a_n, b_n, c_n \mid n \in \mathbb{N}_0 \} \subseteq \mathcal{A}(M)$. Observe that, for each $n \in \mathbb{N}_0$, the generator $b_n$ is the only generator with negative $p_{3n+1}$-adic valuation (which is negative by virtue of the inequality \eqref{eq:I}). As a consequence, we infer that $b_n$ is an atom of $M$. Similarly, for each $n \in \mathbb{N}_0$, the generator $c_n$ is the only generator with negative $p_{3n+2}$-adic valuation, and so $c_n$ is also an atom of $M$. To argue that $a_n \in \mathcal{A}(M)$ for every $n$, fix $j \in \mathbb{N}_0$, and write
	\[
		a_j = \sum_{i=0}^N \alpha_i a_i + \sum_{i=0}^N \beta_i b_i + \sum_{i=0}^N \gamma_i c_i
	\]
	for some $N \in \mathbb{N}$ and coefficients $\alpha_i, \beta_i, \gamma_i \in \mathbb{N}_0$ for every $i \in \ldb 0, N \rdb$. Suppose, for the sake of a contradiction, that there exists $k \in \ldb 0,N \rdb$ such that $\beta_k \neq 0$. The fact that the $p_{3k+1}$-adic valuation of $a_j$ is $0$, along with the inequality $p_{3k+1} > \mathsf{n}\big( \frac45 - \sum_{i=0}^k \frac1{p_{3i}} \big)$, guarantees that $\beta_k$ is a multiple of $p_{3k+1}$. As a result,
	\[
		\frac1{p_{3j}} = a_j > \beta_k b_k >  \frac45 - \sum_{i=0}^k \frac1{p_{3i}} > \frac45 - \frac17 > \sum_{n=0}^\infty \frac{1}{p_n},
	\]
	which is a contradiction. Hence $\beta_i = 0$ for every $i \in \ldb 0, N \rdb$. In a similar manner, we can argue that $\gamma_i = 0$ for every $i \in \ldb 0, N \rdb$. Lastly, applying $p_{3i}$-adic valuation to both sides of the equality $a_j = \sum_{i=0}^N \alpha_i a_i$ for every $i \in \ldb 0, N \rdb \setminus \{j\}$, we obtain that $\alpha_i = 0$ for any $i \in \ldb 0,N \rdb \setminus \{j\}$, and so $a_j$ is an atom of $M$. Hence $\{a_n, b_n, c_n \mid n \in \mathbb{N}_0 \} \subseteq \mathcal{A}(M)$, and the claim is established. Hence the Puiseux monoid $M$ is atomic with
	\[
		\mathcal{A}(M) = \{a_n, b_n, c_n \mid n \in \mathbb{N}_0\}.
	\]
	
	Finally, we claim that $M$ is not a $2$-MCD monoid. Proving this amounts to showing that $\frac{4}{5}$ and $\frac{6}{7}$ have no maximal common divisor in~$M$. Observe that from the definition of $M$, we obtain that $\frac45 = p_1 b_0 + a_0 \in M$, and in the same way we can see that $\frac67 \in M$. To show that $\frac{4}{5}$ and $\frac{6}{7}$ have no maximal common divisor in~$M$, write $\frac{4}{5} = q+d$ and $\frac 67 = r+d$ for some $q,r,d \in M$.
	\[
		\frac45 = q+d = \sum_{i=0}^{N'} \alpha'_i a_i + \sum_{i=0}^{N'} \beta'_i b_i + \sum_{i=0}^{N'} \gamma'_i c_i
	\]
	for some $N' \in \mathbb{N}$ and coefficients $\alpha'_i, \beta'_i, \gamma'_i \in \mathbb{N}_0$ for every $i \in \ldb 0, N' \rdb$. As $\frac45$ has negative $5$-adic valuation, $\beta_k \neq 0$ for some $k \in \ldb 0,N' \rdb$.
	
	Now the fact that the $p_{3k+1}$-adic valuation of $\frac45$ is $0$, in tandem with the inequality $p_{3k+1} > \mathsf{n}\big( \frac45 - \sum_{i=0}^k \frac1{p_{3i}} \big)$, implies that $\beta_k$ is a multiple of $p_{3k+1}$. Thus, any factorization of $\frac45$ contains at least $p_{3k+1}$ copies of the atom $b_k$ for some $k \in \mathbb{N}_0$. In a completely similar way, we can argue that every factorization of $\frac67$ must contain at least $p_{3\ell+2}$ copies of the atom $c_\ell$ for some $\ell \in \mathbb{N}_0$.
	
	We proceed to argue that $d \in \langle a_n \mid n \in \mathbb{N}_0 \rangle$. Suppose, by way of contradiction, that a factorization of $d$ contains an atom of the form $b_j$ for some $j \in \mathbb{N}_0$. This will ensure the existence of a factorization $z'$ of $\frac67 = r+d$ containing at least a copy of the atom $b_j$ and, therefore, at least $p_{3j+1}$ copies of the atom $b_j$ (here we are using once again the inequality $p_{3j+1} > \mathsf{n}\big( \frac45 - \sum_{i=0}^j \frac1{p_{3i}} \big)$). However, as we have seen in the previous paragraph, $z'$ must also contain at least $p_{3\ell+2}$ copies of the atom $c_\ell$ for some $\ell \in \mathbb{N}_0$. Therefore
	\[
		\frac67 \ge p_{3j+1} b_j + p_{3\ell+1} b_\ell = \bigg( \frac45 - \sum_{i=0}^j \frac1{p_{3i}} \bigg) + \bigg( \frac67 - \sum_{i=0}^\ell \frac1{p_{3i}} \bigg) > \frac67,
	\]
	which is a contradiction (the last inequality follows from the fact that $\sum_{n=0}^\infty \frac1{p_n} < \frac17$). Thus, no factorization of $d$ in $M$ contains an atom in $\{b_n \mid n \in \mathbb{N}_0\}$. In a similar way, we can show that no factorization of $d$ in $M$ contains an atom in $\{c_n \mid n \in \mathbb{N}_0 \}$. Since $M$ is atomic, $d \in \langle a_n \mid n \in \mathbb{N}_0 \rangle$, as desired.
	
	Now suppose that $z := z_q + z_d$ is a factorization of $\frac45$, where $z_q$ is a factorization of $q$ and $z_d$ is a factorization of $d$. As $z$ must contain at least $p_{3k+1}$ copies of the atom $b_k$ for some $k \in \mathbb{N}_0$, the fact that $z_d$ cannot contain any copy of the atom $b_k$ ensures that $z_q$ contains at least $p_{3k+1}$ copies of the atom $b_k$. Thus, $p_{3k+1} b_k = \frac45 - \sum_{i=0}^k \frac1{p_{3i}}$ divides $q$ in $M$. Now for each $m > k$, the fact that $a_m$ divides $\frac45 - \sum_{i=0}^k \frac1{p_{3i}}$ in $M$ guarantees that $a_m$ divides $q$ in $M$. Similarly, we can show that $p_{3\ell + 2} c_\ell = \frac67 - \sum_{i=0}^\ell \frac1{p_{3i}}$ divides $r$ in $M$ for some $\ell \in \mathbb{N}_0$, and so that for each $m > \ell$, the element $a_m$ divides $r$ in $M$. From the last two assertions, we infer that if we take $m \in \mathbb{N}$ with $m > \max\{k,\ell\}$, then the atom $a_m$ is a common divisor of both $q$ and $r$ in $M$, which means that $d$ is not a maximal common divisor of $\frac45$ and $\frac67$ in $M$. Hence $\frac45$ and $\frac67$ do not have a maximal common divisor in $M$, and so we conclude that $M$ is not a $2$-MCD monoid.
\end{example}

\begin{rem}
	According to \cite[Proposition 3.2(v)]{FT18}, when a monoid $M$ is cancellative, if the power monoid of $M$ is atomic, then $M$ must be atomic. In light of Example~\ref{ex:atomic PM not 2-MCD}, we conclude that the converse of this statement does not hold.
\end{rem}

In the following result we prove that, unlike atomicity, the ACCP ascends from a Puiseux monoid to its power monoid.

\begin{prop} \label{prop:ACCP ascends from PM to PM}
	If a Puiseux monoid $M$ satisfies the ACCP, then $\ppp_{\emph{fin}}(M)$ also satisfies the ACCP.
\end{prop}

\begin{proof}
	For simplicity, set $\ppp :=  \ppf(M)$. Now let $(B_n + \ppp)_{n \in \nn_0}$ be an ascending chain of principal ideals of~$\ppp$ and let us prove that it stabilizes. From Lemma~\ref{lem:size of the sum} we know that there exists $N \in \nn$ such that $|B_n| = |B_N|$ for every $n > N$. If $|B_N| = 1$, then $(B_n + \ppp)_{n \in \nn_0}$ stabilizes since $M$ satisfies the ACCP. Assume then that $|B_N| > 1$. Let $(C_n)_{n \ge N}$ be a sequence of elements of $\ppp$ satisfying that $C_{n + 1} + B_{n + 1} = B_{n}$ for every $n \ge N$. From Lemma~\ref{lem:size of the sum} we know that $|C_n| = 1$ for every $n > N$, and then we can write $C_n = \{c_n\}$, for some $c_n \in M$. Let $b_{1,n}, b_{2,n}, \dots, b_{k,n}$, be the elements of $B_n$ listed in increasing order. Observe that $c_{n + 1} + b_{i,n + 1} = b_{i,n}$ for every $n \ge N$. Then we have that $(b_{i,n} + M)_{n \ge N}$ is an ascending chain of principal ideals of $M$ for every $i \in \ldb 1, k\rdb$. Since $M$ satisfies the ACCP, every one of those chains of principal ideals stabilizes. Let $t_i$ be a positive integer such that $b_{i,n} = b_{i,t_i}$ for every  $i \in \ldb 1, k\rdb$ and $n > t_i$. Set $t := \max\{t_i \mid 1\le i \le k\}$. Then $B_n = B_t$ for every $n > t$, and therefore, $(B_n + \ppp)_{n \in \nn}$ stabilizes. As a result, the power monoid $\ppp$ satisfies the ACCP.
\end{proof}

We conclude this section by presenting an example of a power monoid of a Puiseux monoid that is atomic but does not satisfy the ACCP. From the previous proposition it is clear that the corresponding Puiseux monoid is not ACCP. 

\begin{example}
	Take $r \in \qq \cap (0,1)$ with $\mathsf{n}(r) \ge 2$, and consider the Puiseux monoid $M_r := \langle r^n \mid n \in \nn_0 \rangle$. It follows from \cite[Theorem~2.3]{CGG20} that the monoid $M_r$ is atomic with $\mathcal{A}(M_r) = \{r^n \mid n \in \nn_0\}$, and it is well known that $M_r$ does not satisfy the ACCP. In fact, observe that the identity
	\[
		x_n := \frac{ \mathsf{n}(r)^{n}}{\mathsf{d}(r)^{n - 1} } = \frac{\mathsf{n}(r)^{n + 1} }{\mathsf{d}(r)^{n} } + \frac{ \big( \mathsf{n}(r) - \mathsf{d}(r) \big) \mathsf{n}(r)^{n}}{\mathsf{d}(r)^{n} }
	\]
	is true for every $n \in \nn$, which yields that the ascending chain of principal ideals $(x_n + M_r)_{n \in \nn}$ does not stabilize. In addition, it was proved in \cite[Example~$4.3$]{GPx} that $M_r$ is an MCD-monoid. Now set $\mathcal{P} := \ppf(M_r)$ and $X_n := \{ x_n\}$ for every $n \in \nn$. Observe that the power monoid $\mathcal{P}$ does not satisfy the ACCP either because $(X_n + \mathcal{P})_{n \in \nn}$ is an ascending chain of principal ideals of $\mathcal{P}$ that does not stabilize. Finally, the fact that $M_r$ is atomic and an MCD-monoid, in tandem with Theorem \ref{thm:ascend of atomicity}, ensures that $\mathcal{P}$ is an atomic monoid. Then $\mathcal{P}$ is a power monoid of a Puiseux monoid that is atomic but does not satisfies the ACCP.
\end{example}

\bigskip
%%%%%%%%%%%%%%%%%%%%%%%%%%%%%%%
%%%%%%%%%%%%%%%%%%%%%%%%%%%%%%%
\section{The Bounded and Finite Factorization Properties}
\label{sec:BF and FF}

The main purpose of this section is to study the bounded and finite factorization properties in the class of power monoids of Puiseux monoids, giving special attention to the ascent of these properties from a Puiseux monoid to its power monoid. Before focusing on the class consisting of all power monoids of Puiseux monoids, let us quickly argue that the restricted power monoid of every Puiseux monoid is an FFM. Although this result is known in greater generality (see \cite[Remark~5.5]{CT23}), for the sake of completeness we offer here an elementary and short inductive proof.

\begin{prop}
	The restricted power monoid of a Puiseux monoid is an FFM.
\end{prop}

\begin{proof}
	Let $M$ be a Puiseux monoid, and set $\ppp := \ppx(M)$. Fix $X \in \ppp$ with $|X| \ge 2$, and let us show that $X$ is atomic in $\ppp$. We proceed by (strong) induction on the cardinality of $X$. The case $|X| = 2$ is clear. Take $k \in \nn$ with $k \ge 2$, and assume that each element of $\ppp$ whose size belongs to $\ldb 2,k \rdb$ is atomic. Suppose now that $|X| = k+1$. If $X$ is not an atom, then there exist $A,B \in \ppp$ with $\max \{|A|, |B|\} < |X|$ such that $X=A+B$ and, as $|A| < |X| = k+1$ and $|B| < |X| = k+1$, our inductive hypothesis ensures that both $A$ and $B$ are atomic, whence $X$ must be atomic.
	
	Now fix $Y \in \ppp^\bullet$. Since every atom that divides $Y$ in $\ppp$ is a subset of $Y$, there are only finitely many atoms dividing $Y$ in $\ppp$. Also, if $A$ is an atom dividing $Y$ in $\ppp$, then no element of the form $m A$ can divide $Y$ when $m > \max Y$. Thus, $Y$ has only finitely many factorizations. Hence $\ppp$ is an FFM, as desired.
\end{proof}

Here we present an example of a Puiseux monoid whose restricted power monoid is neither an HFM nor an LFM.

\begin{example}
	Consider the power monoid $\ppp := \ppx(\nn_0)$. Observe that
	\[
		\{ 0,1 \} + \{ 0,1 \} + \{ 0,1 \} \quad \text{and} \quad \{ 0,1 \} + \{ 0, 2\}
	\]
	are two factorizations of the element $\{0, 1, 2, 3\}$ having different lengths. Hence $\ppp$ is not an HFM. On the other hand, observe that
	\[
		\{ 0,1 \} + \{ 0,2 \} + \{ 0,2 \} \quad \text{and} \quad \{ 0,1 \} + \{ 0, 1\} + \{ 0, 3\}
	\]
	are two different factorizations of the element $\{0, 1, 2, 3, 4, 5\}$ having the same length. As a consequence, $\ppp$ is not an LFM.
\end{example}

We proceed to argue that both the bounded and the finite factorization properties ascend from any Puiseux monoid to its power monoid.

\begin{theorem}\label{BF and FF asscend}
	For a Puiseux monoid $M$, the following statements hold.
	\begin{enumerate}
		\item  If $M$ is a BFM, then $\mathcal{P}_{\emph{fin}}(M)$ is also a BFM.
		\smallskip
		
		\item  If $M$ is an FFM, then $\mathcal{P}_{\emph{fin}}(M)$ is also an FFM.
	\end{enumerate}
\end{theorem}

\begin{proof}
	Set $\ppp := \ppf(M)$.
	\smallskip
	 
	(1) Assume that $M$ is a BFM. Now suppose, towards a contradiction, that the power monoid $\ppp$ is not a BFM. Since $M$ is a cancellative BFM, it must satisfy the ACCP and, therefore, it follows from Proposition~\ref{prop:ACCP ascends from PM to PM} that $\ppp$ also satisfies the ACCP. Thus, from the fact that $\ppp$ is a unit-cancellative monoid satisfying the ACCP, one infers that $\ppp$ is atomic.
	
	As $\ppp$ is atomic but not a BFM, there must exist $B \in \ppp$ such that $\mathsf{L}(B)$ is not finite. Since $M$ is a BFM, we see that $|B| \ge 2$. Fix a nonzero $b \in B$, and then take $\ell \in \nn$ such that $\max \mathsf{L}_M(b) < \ell$ (such an $\ell$ exists because $M$ is a BFM). From Lemma \ref{lem:size of the sum} we know that every factorization in $\mathsf{Z}(B)$ has at most $|B|$ atoms of cardinality greater than $1$. This, along with the fact that $\mathsf{L}(B)$ is not finite, guarantees the existence of a factorization in $\mathsf{Z}(B)$ having at least $\ell$ atoms of cardinality $1$, namely,
	\[
		(\{a_1\} + \{a_2\} + \dots + \{a_\ell\}) + Z_\ell \in \mathsf{Z}(B),
	\]
	where $\{a_1\}, \{a_2\}, \dots, \{a_\ell\}$ are atoms in $\mathcal{A}(\ppp)$ for some $a_1, a_2, \dots, a_\ell \in M^\bullet$ and $Z_\ell$ is a factorization of $C := B - (\{a_1\} + \{a_2\} + \dots + \{a_\ell\})$ in $\ppp$. Now we can take $c \in C$ such that $b = (a_1 + a_2 + \dots + a_\ell) + c$. This equality, along with the fact that $c$ is an atomic element in $M$, yields a factorization of $b$ in $M$ with length at least $\ell$, which contradicts that $\max \mathsf{L}_M(b) < \ell$.
	\smallskip
	
	(2) Assume now that $M$ is an FFM. In particular, $M$ is a BFM, and so it follows from part~(1) that the power monoid $\ppp$ is also a BFM. In order to prove that $\ppp$ is an FFM it suffices to show that $\ppp$ is an LFFM. To argue this, we fix an element $B$ in $\ppp$ and a positive integer $\ell$, and proceed to verify that $B$ has only finitely many factorizations of length $\ell$. Set
	\[
		C := \{c \in M \mid c \text{ divides } b \text{ for some }b \in B\}.
	\]
	As $M$ is a reduced FFM, it follows from \cite[Corollary~2]{fHK92} that every element of $M$ has only finitely many divisors. This, in tandem with the fact that $B$ is a finite set, guarantees that $C$ is a finite set. Now observe that a factorizations of $B$ of length~$\ell$ in the power monoid $\ppp$ can be thought of as a multisets of cardinality $\ell$ consisting of nonempty subsets of $C$, and it is clear that there are only finitely many of such multisets (indeed, at most the number of $\ell$-tuples of nonempty subsets of $C$). This, along with the fact that $\ppp$ is atomic, allows us to conclude that $\ppp$ is an LFFM. Hence the power monoid $\ppp$ is also an FFM, as desired.
\end{proof}

We conclude this section showing that, unlike the bounded and finite factorization properties, the length-finite factorization property does not ascend from a Puiseux monoid to its power monoid. In order to do so, we reuse the monoid constructed in Example~\ref{ex:atomic PM not 2-MCD}.

\begin{prop}
	There exists a Puiseux monoid that is an LFFM whose power monoid is not an LFFM.
\end{prop}

\begin{proof}
	First, we argue that if $(a_n)_{n \ge 1}$ and  $(b_n)_{n \ge 1}$ are two co-well-ordered sequences consisting of positive rationals, then the sequence $(c_n)_{n \ge 1}$ defined as $c_{2n} := a_n$ and $c_{2n-1} = b_n$ for every $n \in \nn$ is also co-well-ordered. Suppose to the contrary that the sequence $(c_n)_{n \ge 1}$ is not co-well-ordered. Then we can find a strictly increasing subsequence $(c'_n)_{n \ge 1}$ of $(c_n)_{n \ge 1}$. Since the terms of $(c'_n)_{n \ge 1}$ are terms of either $(a_n)_{n \ge 1}$ or $(b_n)_{n \ge 1}$, the sequence $(c'_n)_{n \ge 1}$ must contain a subsequence $(c''_n)_{n \ge 1}$ that is also a subsequence of either $(a_n)_{n \ge 1}$ or $(b_n)_{n \ge 1}$. Since $(c'_n)_{n \ge 1}$ is strictly increasing, so is $(c''_n)_{n \ge 1}$. Now if $(c''_n)_{n \ge 1}$ is a subsequence of $(a_n)_{n \ge 1}$, then we obtain a contradiction with the fact that $(a_n)_{n \ge 1}$ is a co-well-ordered sequence. If $(c''_n)_{n \ge 1}$ is a subsequence of $(b_n)_{n \ge 1}$, then we obtain a similar contradiction.
	
	By the argument given in the previous paragraph, we infer that the sum of two Puiseux monoids that are co-well-ordered is again a co-well-ordered, which inductively implies that the sum of finitely many co-well-ordered Puiseux monoids is again co-well-ordered. Now let $M$ be the Puiseux monoid constructed in Example~\ref{ex:atomic PM not 2-MCD}; that is,
	\[
		M := \bigg\langle \frac{1}{p_{3n}}, \ \frac{1}{p_{3n+1}} \bigg( \frac45 - \sum_{i=0}^n \frac1{p_{3i}} \bigg), \frac{1}{p_{3n+2}} \bigg( \frac67 - \sum_{i=0}^n \frac1{p_{3i}} \bigg) \ \Big{|} \ n \in \mathbb{N}_0 \bigg\rangle.
	\]
	The monoid $M$ is generated by the union of the underlying set of three decreasing sequences, then $M$ can be naturally expressed as the sum of three co-well-ordered Puiseux monoids. Thus, $M$ is a co-well-ordered Puiseux monoid. We have proved in Example~\ref{ex:atomic PM not 2-MCD} that $M$ is atomic, and so it follows from \cite[Theorem~3.4]{JKK23} that $M$ is an LFFM. On the other hand, we have also argued in Example~\ref{ex:atomic PM not 2-MCD} that $M$ is not $2$-MCD, and so it follows from part~(2) of Theorem~\ref{thm:ascend of atomicity} that the power monoid of $M$ is not atomic. Hence the power monoid of $M$ is not an LFFM even though $M$ is an LFFM.
\end{proof}

\bigskip
%%%%%%%%%%%%%%%
%%%%%%%%%%%%%%%
\section*{Acknowledgments}

The authors are grateful to their mentor, Dr. Felix Gotti, for proposing this project and for his guidance during the preparation of this paper. While working on this paper, the authors were part of CrowdMath, a massive year-long math research program hosted by MIT PRIMES and the Art of Problem Solving (AoPS). The authors express their gratitude to the directors and organizers of CrowdMath, MIT PRIMES, and AoPS for making this research experience possible.

\bigskip
%%%%%%%%%%%%%%%

\end{document}